\documentclass[11pt]{article}

\usepackage{hyperref}
\usepackage{amsmath,amssymb,amstext}
\usepackage{amsthm}
\usepackage[utf8]{inputenc}
\usepackage{subcaption} 
\usepackage{tikz}
\usepackage{fancyhdr}
\usepackage{titlesec}
\usepackage{mathtools}

\usetikzlibrary{calc}
\usetikzlibrary{patterns}

\newcommand{\E}{\mathbb{E}}
\newcommand{\R}{\mathbb{R}}
\newcommand{\N}{\mathbb{N}}

\newcommand{\bd}{\operatorname{bd}}

\renewcommand{\phi}{\varphi}

\usepackage[T1]{fontenc}

\def\co#1{{\text{\normalfont conv}}[#1]}

\newtheorem{lemma}{Lemma}

\newtheorem{prop}{Proposition}
\newtheorem{theorem}{Theorem}

\titleformat{\section}{\normalfont\bfseries}{\thesection}{1em}{\uppercase}

\begin{document} 
\begin{center} \LARGE 
Monotonicity of the Sample Range of 3-D Data:\\ 
Moments of Volumes of Random Tetrahedra
\normalfont\normalsize

\vspace{1cm}

Stefan Kunis, Benjamin Reichenwallner and Matthias Reitzner 

\vspace{0.1cm}

\textit{University of Osnabrueck and University of Salzburg}

\vspace{0.2cm}

\rule{5cm}{0.01cm}

\end{center}

\abstract{The sample range of uniform random points $X_1, \dots , X_n$ chosen in a given convex set 
is the convex hull $\co{X_1, \dots, X_n}$. It is shown that in dimension three the expected 
volume of the sample range is not monotone with respect to set inclusion. This 
answers a question by Meckes in the negative.

The given counterexample is the three-dimensional tetrahedron together with an infinitesimal variation of it.
As side result we obtain an explicit formula for all even moments of the volume of a random simplex 
which is the convex hull of three uniform random points in the tetrahedron and the center of one 
facet.
}


\bigskip

\section{Introduction}
Choose random points $X_1, \dots, X_n$ independently according to the uniform distribution in an 
interval $I \subset \R$. The convex hull $\co{X_1, \dots , X_n}$ of these random points is the 
well-known sample range which can also be defined as the interval $[X_{[1]}, X_{[n]}]$, where 
$X_{[1]} \leq \dots \leq X_{[n]}$ is the the order statistic of the random points, and the 
endpoints $X_{[1]},  X_{[n]}$ are the extreme points of the random sample. It is trivial and 
immediate that the expected 
length of the sample range is a monotone function in $I$: Choosing uniform random points $Y_1, 
\dots, Y_n$ in an interval $J \subset I$, one has 
\begin{equation}\label{eq:mon-d=1}
 \E |\co{Y_1, \dots, Y_n}| \leq  \E |\co{X_1, \dots, X_n}|,  
\end{equation}
where $|A|$ is the Lebesgue measure of the set $A$.

A generalization of this question to higher dimensions leads to nontrivial problems: First the 
definition of sample range, order statistic and extreme points is not obvious. Maybe the most 
natural extension of {\it extreme points} and {\it sample range} for higher dimensions is the 
following:

\begin{quote}
For random points $X_1, \dots , X_n \in \R^d$, we define the sample range to be the convex hull 
$\co{X_1, \dots, X_n}$, and the extreme points of the sample are those on the boundary of the 
sample 
range, i.e., the vertices of $\co{X_1, \dots, X_n}$. 
\end{quote}

The question we want to adress in this paper is the following: {\it Is the expected volume of the 
sample range a monotone function in the underlying distribution?} To make this question more 
precise, we assume (as the most simple example) that the points are chosen according to the uniform 
measure in a convex set $K$. Then the monotonicity question \eqref{eq:mon-d=1} reads as follows:

\noindent
Assume that $L, K$ are two $d$-dimensional convex  sets. Choose independent uniform random points 
$Y_1, \dots, Y_n$ in $L$ and  
$X_1, \dots , X_n$ in $K$. Is it true that $L \subset K$ implies
\begin{equation}\label{eq:mon-d-gen}
 \E | \co{Y_1, \dots, Y_n} | \leq  \E | \co{X_1, \dots, X_n}|?  
\end{equation}
Here, $|A|$ denotes the $d$-dimensional Lebesgue measure of the $d$-dimensional set $A$.
The starting point for these investigations should be a check for the first nontrivial case 
$n=d+1$ where the sample range is the random simplex spanned by the sample points. In this form, the 
question was first raised by Meckes \cite{meck} in the context of high-dimensional convex 
geometry.

As already mentioned, in dimension one this is immediate. It was proved by Rademacher \cite{rad} 
in 2012 that this is also true in dimension two. Our main result solves the three-dimensional case.
\begin{theorem}\label{Th1}
In $\R^3$ the expected volume of the sample range is in general {\bf not} monotone in the 
underlying distribution. There are three-dimensional  convex sets $L \subset K$ such that 
\begin{equation}
 \E | \co{Y_1, \dots, Y_{4}} | >  \E | \co{X_1, \dots, X_{4}}|,
\end{equation}
if $Y_1, \dots, Y_{4}$ are chosen uniformly in $L$ and $X_1, \dots, X_{4}$ in $K$.
\end{theorem}
 
That the general question \eqref{eq:mon-d-gen} cannot be answered in the positive was already shown 
by Rademacher who, in a groundbreaking paper, gave counterexamples for dimensions $d \geq 4$ and 
$n=d+1$. 
It remains an open problem whether there is a number $N$, maybe depending on $K$ or only on the 
dimension of the underlying space, such that monotonicity holds for $n \geq N$. And a suitable 
precise formulation of the question for non-uniform measures would also be highly interesting.

\bigskip
For our proof, we need to construct a pair of convex sets leading to a counterexample. A serious 
drawback of this approach is that one is forced to compute the expected volume of a 
random simplex which is known to be a notorious hard problem. In dimension two, 
tedious but explicit computations from the nineteenth century yielded several 
explicit results, but starting with dimension three, the problem turns out to be out of reach in 
general. The only three-dimensional convex sets where the expected volume of a 
random simplex is known are the ball \cite{miles}, the cube \cite{zinani} and the 
tetrahedron \cite{bucRei}. And in higher dimensions only the ball allows for explicit results.
Since numerical computations in dimension three suggest that in the neighbourhood of the cube and 
the ball the expected volume of a random simplex is monotone, the only potentially tractable 
counterexample could be the tetrahedron and a set close to it, which also is in accordance with 
numerical computations by Rademacher\cite{rad}, and Reichenwallner and Reitzner\cite{self}.

Already the determination of the expected volume of a random simplex in a tetrahedron $T \subset \R^3$ was 
extremely hard. This question is known as Klee's problem, and after many attemps, erroneous 
conjectures and numerical estimates, Reitzner and Buchta \cite{bucRei} proved in a long paper that for uniform random points $X_1,\ldots,X_4$ in a tetrahedron of volume one, we have
\begin{equation}\label{eq:bucrei}
\E | \co{X_1, \dots, X_{4}} |    = \frac{13}{720}- \frac{\pi^2}{15\,015} = 0.01739\ldots 
\end{equation}       
                                                          
It seems to be out of reach to compute this expectation for any other three-dimensional convex set 
close to $T$. Luckily there is a wonderful alternative approach due to Rademacher, using an 
infinitesimal variation of convex sets, which is stated in the following Lemma. 

\begin{lemma}[Rademacher~\cite{rad}]\label{equ}
	For $d \in \N$, monotonicity under inclusion of the map 
	$$K \mapsto \E | \co{X_1, \dots, X_{d+1}} |  ,$$ 
	where $K$ ranges over all $d$-dimensional convex bodies and $X_i$ are iid uniform points 
in $K$, holds if and only if we have for each convex body $K \subseteq \mathbb{R}^d$ and for each 
$z \in \bd K$ that 
	$$ \E | \co{X_1, \dots, X_{d+1}} |   \leq \E | \co{X_1, \dots, X_{d},z} |  .$$
\end{lemma}
Hence we get the counterexample for Theorem~\ref{Th1} if we succeed in computing the expectation $ \E | \co{X_1, \dots, X_{3},z} |$ for some $z 
\in \bd T$. Because of symmetry, a suitable 
choice for $z$ should be the center $c$ of one of the facets. Yet after several attempts, we observed 
that computing $ \E | \co{X_1, \dots, X_{3},c} |$ is even more difficult than \eqref{eq:bucrei} and 
hence impossible.
Neverless we will prove the following proposition.
\begin{prop}\label{prop:main}
For a tetrahedron $T$ of volume one, $c$ the centroid of a facet of $T$ and $X_1,\ldots,X_3$ uniform random points in $T$, we have that 
	$$ \E | \co{X_1, \dots, X_{3},c} |  < \frac{13}{720}- \frac{\pi^2}{15\,015}  =
	\E | \co{X_1, \dots, X_{4}} |    .$$
\end{prop}

A combination of this result with Rademacher's Lemma \ref{equ} yields Theorem \ref{Th1}.
The rigorous bound in Proposition \ref{prop:main} is obtained by combining methods from stochastic 
geometry with 
results from approximation theory. In the background, first there is a result about the 
precise approximation of the absolute value function on $[ -\frac 13, \frac 13]$ by suitable 
even polynomials, Lemma \ref{lem:interp}. To apply this in our context, we use an explicit result for 
{\it all} even moments of $| \co{X_1, \dots, X_{3},c} |$ which --- at a first glance maybe 
surprisingly --- is much easier to obtain then just the single first moment.

\begin{theorem}\label{Th2}
Let $k \in \mathbb{N}$ and choose three uniform random points $X_1,\ldots,X_3$ in a tetrahedron of volume one. Then it holds:
\begin{align*}
\E | \co{X_1, \dots, X_{3},c} |^{2k} 
&=   
\frac {8 }{3^{2k-3} }
\sum_{\sum_1^{18}k_i=2k} (-1)^{k^\prime} 3^{k^{\prime\prime}} 
\binom{2k}{k_1,\ldots,k_{18}}\, \prod_{i=1}^3 \frac{l_i!\,m_i!\,n_i!}{(l_i+m_i+n_i+3)!},
\end{align*}
where the range of summation and abbreviations are given in \eqref{eq:def-summ}.
The first five even moments are given at the end of Section~\ref{sec:even}.
\end{theorem}

\bigskip
This paper is organized in the following way. In Section~\ref{sec:even}, we give a series representation for even moments of the volume of a random tetrahedron inside a tetrahedron where one point is fixed to be the centroid of a facet, and we use that to find an exact value for the first thirteen even moments. In Section~\ref{sec:proof}, we compute an upper bound for the expected volume of our random tetrahedron, which is a rational affine combination of those even moments. This upper bound suffices to show that the tetrahedron is a counterexample.

As a general reference for results on random polytopes, we refer to the book on Stochastic and Integral Geometry by Schneider and Weil \cite{schWeil}. More recent surveys are due to Hug \cite{hug_surv} and Reitzner \cite{reitz_surv}.

\section{Even Moments of the Volume of Random Simplices}\label{sec:even}

Let $T$ be a tetrahedron of volume one and $c = (x_c,y_c,z_c)$ the centroid of one of its facets. 
For random points $X_1,X_2, X_3 \in T$, we write $X_i = (x_i,y_i,z_i), i = 1,2,3$.
The volume of the simplex with vertices $X_1,X_2, X_3$ and $c$ is given by
\begin{align*}
| \co{X_1, \dots, X_{3},c} |   &=  \left|\frac{1}{6}\det\begin{pmatrix}
x_1 & y_1 & z_1 & 1\\
x_2 & y_2 & z_2 & 1\\
x_3 & y_3 & z_3 & 1\\
x_c & y_c & z_c & 1
\end{pmatrix}\right|
= 6^{-1} | D(x_1, \dots , z_c)| ,
\end{align*}
and hence by the absolute value of a polynomial $D$ of degree precisely three in the coordinates of 
$X_1,X_2,X_3$ and $c$. We are interested in the even moments of $ | \co{X_1, \dots, X_{3},c} |$, 
where we get rid of the absolute value.
\begin{align*}
\E | \co{X_1, \dots, X_{3},c} |^{2k} 
&= 
6^{-2k} \int\limits_T \int\limits_T \int\limits_T \ D(x_1, \dots, z_c)^{2k} \ 
\\ & \hskip3.5cm 
d(x_1,y_1,z_1)\, d(x_2,y_2,z_2)\, d(x_3,y_3,z_3).  
\end{align*}

Let $T_o$ be the specific tetrahedron 
$$ T_o = \{(x,y,z)\in \mathbb{R}^3: x,y,z \geq 0, x+y+z\leq 
1 \},$$ 
i.e., that with vertices $(0,0,0), (1,0,0)$, $(0,1,0)$ and 
$(0,0,1)$. Note that the volume of $T_o$ is $ 1/6$. We choose $c_o=(1/3,1/3,0)$, the centroid 
of the facet $\{(x,y,0)\in\mathbb{R}^3: x,y\geq 0, x+y \leq 1\}$. 

Since the expectation $\E | \co{X_1, \dots, X_{3},c} |$ is invariant under volume-preserving affine transformations, we 
can use as a representative of a tetrahedron of volume one the tetrahedron ${\sqrt[3] 6}\, T_o $ and 
the center ${\sqrt[3]6}\, c_0$. We have:
\begin{align}\label{eq:intpol}
&\E | \co{X_1, \dots, X_{3},c} |^{2k} \nonumber\\
&\hspace{2cm}=  \nonumber
6^{-2k} 
\int\limits_{{\sqrt[3] 6} T_o} \int\limits_{{\sqrt[3] 6} T_o} \int\limits_{{\sqrt[3] 6} T_o} 
\ D(x_1, \dots, z_{{\sqrt[3]6 }c_0})^{2k} \  
\\ & \nonumber \hskip6.7cm  
d(x_1,y_1,z_1)\,d(x_2,y_2,z_2)\,d(x_3,y_3,z_3)  
\\[1ex] &\hspace{2cm}= 
6^{3} 
\int\limits_{T_o} \int\limits_{T_o} \int\limits_{T_o} 
D(x_1, \dots, z_{c_o})^{2k} \  
d(x_1,y_1,z_1)\,d(x_2,y_2,z_2)\,d(x_3,y_3,z_3).  
\end{align}
Expanding the determinant, the polynomial $D$ can be written as
\begin{align*}
D(x_1, \dots, z_{c_o})
=&
\frac{1}{3} \Big(x_1 z_2 - x_1 z_3 - x_2 z_1 + x_2 z_3 + x_3 z_1 - x_3 z_2 - y_1 z_2 + y_1 z_3 
\\ &\hspace{1cm}
+ y_2 z_1 - y_2 z_3 - y_3 z_1 + y_3 z_2 + 3x_1 y_2 z_3 - 3 x_1 y_3 z_2 
\\ &\hspace{1cm}
- 3x_2 y_1 z_3 + 3x_2 y_3 z_1 + 3x_3 y_1 z_2 - 3x_3 y_2 z_1 \Big).
\end{align*}
By the Multinomial Theorem, and using the multinomial 
coefficient
$$\binom{2k}{k_1,\ldots,k_{18}} = \frac{(2k)!}{k_1! \cdots k_{18}!} , $$
the $(2k)$-th power of it can be rewritten as 
\begin{align}\label{eq:pol2k}
D(x_1, \dots, z_{c_o})^{2k}
=& \nonumber
3^{-2k} \sum_{\sum_1^{18} k_i=2k} (-1)^{k^\prime} 
3^{k^{\prime\prime}} \binom{2k}{k_1,\ldots,k_{18}} (x_1 z_2)^{k_1}\, (x_1 z_3)^{k_2}\, (x_2 z_1)^{k_3}
\\ &  \nonumber
\hspace{1cm}\times (x_2 z_3)^{k_4} \, (x_3 z_1)^{k_5}\, (x_3 z_2)^{k_6}\, (y_1 z_2)^{k_7}\, (y_1 z_3)^{k_8}\, (y_2 z_1)^{k_9} \,
\\ &  \nonumber
\hspace{1cm}\times (y_2 z_3)^{k_{10}}\, (y_3 z_1)^{k_{11}}\, (y_3 z_2)^{k_{12}}\, (x_1 y_2 z_3)^{k_{13}}\, (x_1 y_3 z_2)^{k_{14}}\, 
\\ &  \nonumber 
\hspace{1cm} \times  (x_2 y_1 z_3)^{k_{15}} \,
(x_2 y_3 z_1)^{k_{16}}\, (x_3 y_1 z_2)^{k_{17}}\, (x_3 y_2 z_1)^{k_{18}}
\\ = &
3^{-2k} \sum_{\sum_1^{18}k_i=2k} (-1)^{k^\prime} 3^{k^{\prime\prime}} 
\binom{2k}{k_1,\ldots,k_{18}}\, \prod_{i=1}^3 x_i^{l_i} \, y_i^{m_i} \, z_i^{n_i} \ .
\end{align}
Here for abbreviation we use the following notation:
\begin{align}\label{eq:def-summ}
 \nonumber
k^\prime &=k_2+k_3+k_6+k_7+k_{10}+k_{11}+k_{14}+k_{15}+k_{18}, \\ \nonumber
k^{\prime\prime} &= k_{13}+k_{14}+k_{15}+k_{16}+k_{17}+k_{18}, \\ \nonumber
l_1 &= k_1+k_2+k_{13}+k_{14},\\ \nonumber
m_1 &= k_7+k_8+k_{15}+k_{17},\\ \nonumber
n_1 &= k_3+k_5+k_9+k_{11}+k_{16}+k_{18},\\
l_2 &= k_3+k_4+k_{15}+k_{16},\\ \nonumber
m_2 &= k_9+k_{10}+k_{13}+k_{18},\\ \nonumber
n_2 &= k_1+k_6+k_7+k_{12}+k_{14}+k_{17},\\ \nonumber
l_3 &= k_5+k_6+k_{17}+k_{18},\\ \nonumber
m_3 &= k_{11}+k_{12}+k_{14}+k_{16},\\ \nonumber
n_3 &= k_2+k_4+k_8+k_{10}+k_{13}+k_{15}.\\ \nonumber
\end{align}

Integration of the monomials over the tetrahedron $T_o$ gives
\begin{eqnarray*}
\int\limits_{T_o}  x^{l_i}\, y^{m_i}\, z^{n_i}\  d(x,y,z) 
&=&
\underbrace{\int\limits_0^1\int\limits_0^{1} \int\limits_0^{1}}_{x+y+z \leq 1} 
x^{l_i}\, y^{m_i}\, z^{n_i} \  dx\,dy\,dz
\\ 
\left| \begin{array}{ll}
	z=t \\ y=s(1-t) \\ x=r(1-s)(1-t)
	\end{array}
\right|
&=&
\int\limits_0^1 r^{l_i}\; dr \int\limits_0^{1} s^{m_i}\, 
(1-s)^{l_i+1}\,ds \int\limits_0^{1} t^{n_i}\, (1-t)^{l_i+m_i+2}\,dt 
\\ &=&
\frac1{l_i+1}\, B(m_i+1,l_i+2)\, B(n_i+1, l_i+m_i+3)
\\ &=&
\frac{l_i!\,m_i!\,n_i!}{(l_i+m_i+n_i+3)!}  \ .
\end{eqnarray*}

Combining this with equations \eqref{eq:intpol} and \eqref{eq:pol2k} yields
\begin{align*}
\E | \co{X_1, \dots &, X_{3},c} |^{2k} =
\\ &=   
\frac {8 }{3^{2k-3} }
\sum_{\sum_1^{18}k_i=2k} (-1)^{k^\prime} 3^{k^{\prime\prime}} 
\binom{2k}{k_1,\ldots,k_{18}} \prod_{i=1}^3 \frac{l_i!\,m_i!\,n_i!}{(l_i+m_i+n_i+3)!},
\end{align*}
which is Theorem \ref{Th2}.
We list the first five even moments of the volume of a random simplex in a tetrahedron $T$ of 
volume one:
\begin{align*} \label{eq:evenmoments}
\E | \co{X_1, \dots, X_{3},c} |^{2} &= \frac{1}{2\, 000} = 0.0005,\\
\E | \co{X_1, \dots, X_{3},c} |^4 &= \frac{43}{27\,783\, 000} \approx 1.54771\cdot 10^{-6},\\
\E | \co{X_1, \dots, X_{3},c} |^6 &= \frac{347}{28\,805\,414\,400} \approx 1.20463\cdot 10^{-8},\\
\E | \co{X_1, \dots, X_{3},c} |^8 &= \frac{2\,389}{14\,263\,395\,300\,000} \approx 1.67492\cdot 
10^{-10},\\
\E | \co{X_1, \dots, X_{3},c} |^{10} &= \frac{310483}{90\,249\,636\,885\,408\,000} \approx 
3.44027\cdot 10^{-12}.
\end{align*}

\section{Proof of Theorem 1}\label{sec:proof}

As described in Section~\ref{sec:even}, the $(2k)$-th moment of $| \co{X_1, \dots, X_{3},c} |$ can 
be computed, with fast increasing complexity in $k$. Also note that the volume of a tetrahedron in 
$T$, where one vertex is fixed to be the centroid of a facet of $T$, is not larger than $1/3$. 
Hence, we want to approximate the absolute value function in the interval $[-1/3,1/3]$ by a 
polynomial
 $$P(x) = \sum_{i=0}^n a_i x^{2i}$$
for some $n\in\mathbb{N}$ such that $P(x) > |x|$ for all $x \in [-1/3,1/3]$ or, equivalently, 
\linebreak $P(x) > x$ for all $x \in [0,1/3]$. 
In contrast to the classical problem of {\it best approximation} of $|x|$ by polynomials, we are interested in one-sided approximation
and a certain expected value of the polynomial as objective.
We use the following standard result for polynomial interpolation.

\begin{lemma}\label{lem:interp}
 Let $m\in\N$, $n=2m+1$, and $0<x_0<\hdots<x_m$ be given. Then the system of equations $$P(x_j)=x_j 
\text{ and } P'(x_j)=1 \quad \text{ for } j=0,\hdots,m$$ determines uniquely a polynomial 
$P(x)=\sum_{i=0}^n a_i x^{2i}$ with the property $P(x)\ge |x|$ for all $x\in\R$.
\end{lemma}
\begin{proof}
 Let $t_j=x_j^2$ and consider the standard Hermite interpolation problem 
$$Q(t_j)=f(t_j) \text{ and } Q'(t_j)=f'(t_j) \quad \text{for } j=0,\hdots,m$$ for the functions 
$f(t)=\sqrt{t}$ and $Q(t)=\sum_{i=0}^n a_i t^i$. The condition $P'(x_j)=1$ is equivalent to $Q'(t_j)= 1/(2\sqrt{t_j}) =f'(t_j)$. Then the interpolation error fulfills,
 for some $\xi\in [t_0,t_m]$, the estimate
 \begin{equation*}
  f(t)-Q(t)=\frac{f^{(n+1)}(\xi)}{(n+1)!}\prod_{j=0}^m (t-t_j)^2 < 0,
 \end{equation*}
 where the last inequality follows from $f^{(n+1)}(t)<0$ for all $t$.
\end{proof}

We note in passing that for even $n=2m$ and $0<x_0<\hdots<x_m=1/3$, we only require the simple 
interpolation condition $P(x_m)=x_m$ in the last point and get $P(x)\ge |x|$ for all $x\in 
[-x_m,x_m]$.

\bigskip
Our aim is to approximate $\E | \co{X_1, \dots, X_{3},c} |$ from above by an even polynomial $P$ of 
degree $2n$,
\begin{eqnarray*}
\E | \co{X_1, \dots, X_{3},c} |
\leq 
\E P(| \co{X_1, \dots, X_{3},c} |),
\end{eqnarray*}
which holds if $|x| \leq P(x) $ on $[- \frac 13, \frac 13]$.
Moreover, the best polynomial for fixed $n\in\N$ can be found via the linear optimization problem
\begin{align*}
\min_P \ \E P(| \co{X_1, \dots, X_{3},c} |) &=
\min_{a_i}\  \sum_{i=0}^n a_{i}\, \E P(| \co{X_1, \dots, X_{3},c} |)^{2i}  
\\ & \hspace{3cm} 
\text{s.t.} \ P(x)\ge x, \; x\in \left[0,\frac{1}{3}\right].
\end{align*}
Please note that the constraint is infinite dimensional. Relaxing the constraint, we get a lower 
bound on $\E P(| \co{X_1, \dots, X_{3},c} |)$ via the finite dimensional linear program
\begin{equation*}
 \min_P \E P(| \co{X_1, \dots, X_{3},c} |) \quad \text{s.t.} \  P(x_{\ell})\ge x_{\ell}, \; 
x_{\ell}\in \left[0,\frac{1}{3}\right],\;\ell=0,\hdots,L.
\end{equation*}
For $n=12$ and $L=100$ equidistant points $x_{\ell}\in[0,1/3]$, we numerically compute via Matlab 
and the optimization toolbox CVX \cite{cvx} 
$$ \E P(| \co{X_1, \dots, X_{3},c} |) >0.01746,$$
yielding that we do not get a sufficiently precise estimate using only $n=12$ even moments.

For $n=13$ and $L=1\,000$, we solved the above linear program, computed the interpolation nodes with the absolute value function numerically, and rationalized these points to 
\begin{equation*}
 \{x_j:j=0,\hdots,6=m\} = \left\{\frac{1}{83}, \frac{1}{22}, \frac{1}{11}, \frac{2}{15}, \frac{2}{11}, \frac{5}{22}, \frac{4}{15}\right\}.
\end{equation*}
Using these points for the interpolation problem in Lemma \ref{lem:interp} gives an even 
polynomial $P_{\rm cert}(x)= \sum_{i=0}^{13} a_i x^{2i} $ of degree 26 with explicitly given 
rational coefficients $a_0, \dots, a_{13}$ and the property $|x| \leq P_{\rm cert}(x)$. 
Finally, we use the even moments computed in Section \ref{sec:even} to complete the proof of 
Theorem~\ref{prop:main}:

\begin{align*}
\E | \co{X_1, \dots, X_{3},c} |
&\le 
\E P_{\rm cert}(| \co{X_1, \dots, X_{3},c} |)=
\sum_{i=0}^{13} a_i \E | \co{X_1, \dots, X_{3},c} |^{2i}
\\  &=
\underbrace{
\scalebox{0.52}{$  
\frac{
9215716290354844120429638007455369524673678722793816500930077456473045568887048115406728916177539584
714665872679760706490830685597598152285959465854826727651}{
5302749610595656531552535250244971867247891347471652937573234276503331818475900122041632848466907674
94412592136803582459020919668764372661702456688640000000000}$}
}_{= 0.0173791 \dots}
\\  &<
\underbrace{ \frac{13}{720}- \frac{\pi^2}{15\,015} }_{= 0.01739\ldots}
= 
\E | \co{X_1, \dots, X_{4}} |.   
\end{align*}

\bigskip
\parindent=0pt

\begin{samepage}
	Stefan Kunis, Matthias Reitzner\\
	Institut f\"ur Mathematik\\
	Universit\"at  Osnabr\"uck\\
	Albrechtstr. 28a \\
	49076 Osnabr\"uck, Germany\\
	e-mail: stefan.kunis, matthias.reitzner@uni-osnabrueck.de
\end{samepage}

\bigskip
\begin{samepage}
	Benjamin Reichenwallner\\
	Fachbereich Mathematik \\
	Universit\"at Salzburg\\
	Hellbrunnerstra\ss e 34\\
	5020 Salzburg, Austria\\
	e-mail: benjamin.reichenwallner@sbg.ac.at
\end{samepage}


\begin{thebibliography}{1}
	
	\bibitem{bucRei}
	{C.~Buchta and M.~Reitzner}, { The convex hull of random points in a tetrahedron: Solution of Blaschke's problem and more general results.} 
	{\it J. reine angew. Math.} {\bf 536} (2001), 1--29.
	
	\bibitem{cvx}
	{M.~Grant and S.~Boyd}, { {CVX}: Matlab Software for Disciplined Convex Programming, version 2.1.}, {\url{http://cvxr.com/cvx}}, 2014.
	
	\bibitem{hug_surv}
	{D.~Hug}, { Random polytopes.}
	{In: Spodarev, E. (ed.): {\it Stochastic Geometry, Spatial Statistics and Random Fields}. 
		Lecture Notes in Mathematics {\bf 2068}, pp. 205--238, Springer, Heidelberg} 2013.
	
	\bibitem{meck}
	{M.~Meckes}, { Monotonicity of volumes of random simplices.} 
	In: Recent Trends in Convex and Discrete Geometry, 2006.
	
	\bibitem{miles}
	{R.~E. Miles}, { Isotropic random simplices.} {\it Adv. in Appl. Probab.} {\bf 3}
	(1971), 353--382.


	\bibitem{rad}
	{L.~Rademacher}, { On the monotonicity of the expected volume of a
		random simplex.} {\it Mathematika} {\bf 58} (2012), 77--91.
	
	\bibitem{self}
	{B.~Reichenwallner and M.~Reitzner}, { On the monotonicity of the moments of volumes of random simplices.} {\it Mathematika} {\bf 62}
	(2016), 949--958.
	
	\bibitem{reitz_surv}
	{M.~Reitzner}, { Random polytopes.}
	In: Kendall W.S. and Molchanov I. (eds.): {\it New perspectives in stochastic geometry.} pp. 45--76, Oxford Univ. Press, Oxford, 2010. 
	
	\bibitem{schWeil}
	{R.~Schneider and W.~Weil}, { Stochastic and integral geometry.}
	Probability and its Applications (New York), Springer-Verlag, Berlin, 2008.

	\bibitem{zinani}
	{A.~Zinani}, { The Expected Volume of a Tetrahedron whose Vertices are Chosen at Random 
	in the Interior of a Cube.} {\it Monatsh. Math.} {\bf 139}
	(2003), 341--348.

	\end{thebibliography}
\end{document}